\documentclass[11pt]{amsart}
\usepackage[margin=30mm]{geometry}
\usepackage{amsmath,amssymb}
\usepackage{amsthm}
\usepackage{mathrsfs}

\newtheorem{thm}{Theorem}[section]

\newtheorem{lem}{Lemma}[section]

\newtheorem{defi}{Definition}[section]
\newtheorem{ex}{Example}[section]
\newtheorem{rem}{Remark}[section]

\newtheorem*{ques}{\it{Question}}

\usepackage{enumitem}

\begin{document}

\title{Shadowing, sensitivity and entropy points}
\author{Noriaki Kawaguchi}
\address{Department of Mathematical and Computing Science, School of Computing, Institute of Science Tokyo, 2-12-1 Ookayama, Meguro-ku, Tokyo 152-8552, Japan}
\email{gknoriaki@gmail.com}

\begin{abstract}
For continuous self-maps of compact metric spaces, we explore the relationship among the shadowable points, sensitive points, and entropy points. Specifically, we show that (1) if the set of shadowable points is dense in the phase space, then a point located in the interior of the set of sensitive points is an entropy point; and (2) if the topological entropy is zero, then the denseness of the set of shadowable points is equivalent to almost chain continuity. In addition, we present a counter-example to a question raised by Ye and Zhang regarding entropy points.
\end{abstract}

\maketitle

\markboth{NORIAKI KAWAGUCHI}{Shadowing, sensitivity and entropy points}

\section{Introduction}

{\em Shadowing} is an important concept in the theory of dynamical systems. It was initially introduced in the context of hyperbolic differentiable dynamics \cite{An,B} and generally refers to a situation in which coarse orbits, or {\em pseudo-orbits}, are approximated by true orbits. For background on the shadowing theory, we refer the reader to the monograph \cite{AH}. In \cite{Mor}, by localizing global shadowing into pointwise shadowings, Morales introduced the notion of {\em shadowable points}. In \cite{K4}, a shadowable point that is also an entropy point of a certain type is characterized by the structure of pseudo-orbits. In this paper, we discuss the relationship among the shadowable points, sensitive points, and entropy points. For instance, one of the main results is that if the set of shadowable points is dense in the phase space, then a point located in the interior of the set of sensitive points is an entropy point (Theorem 1.2). Another result is that if the topological entropy is zero, then the denseness of the set of shadowable points is equivalent to almost chain continuity (Theorem 1.3). In contrast to \cite{K4}, a particular focus of this paper is on the terminal chain components. We also present a counter-example to a question raised by Ye and Zhang \cite{YZ} regarding entropy points.

We begin by defining shadowable points. Throughout, $X$ denotes a compact metric space endowed with a metric $d$.

\begin{defi}
\normalfont
Let $f\colon X\to X$ be a continuous map and let $\xi=(x_i)_{i\ge0}$ be a sequence of points in $X$. For $\delta>0$, $\xi$ is called a {\em $\delta$-pseudo orbit} of $f$ if $d(f(x_i),x_{i+1})\le\delta$ for all $i\ge0$. For $\epsilon>0$, $\xi$ is said to be {\em $\epsilon$-shadowed} by $x\in X$ if $d(f^i(x),x_i)\leq \epsilon$ for all $i\ge 0$. We say that $x\in X$ is a {\em shadowable point} for $f$ if for any $\epsilon>0$, there is $\delta>0$ such that every $\delta$-pseudo orbit $(x_i)_{i\ge0}$ of $f$ with $x_0=x$ is $\epsilon$-shadowed by some $y\in X$. We denote by $Sh(f)$ the set of shadowable points for $f$.
\end{defi}

For a topological space $Z$, a subset $S$ of $Z$ is called a $G_\delta$-subset of $Z$ if $S$ is a countable intersection of open subsets of $Z$.

\begin{rem}
\normalfont
\begin{itemize}
\item[(1)] Let $f\colon X\to X$ be a continuous map. For any $j,l\ge1$, let $S_{j,l}$ denote the set of $x\in X$ such that there is an open neighborhood $U$ of $x$ for which every $\frac{1}{j}$-pseudo orbit $(x_i)_{i\ge0}$ of $f$ with $x_0\in U$ is $\frac{1}{l}$-shadowed by some $y\in X$. We see that $S_{j,l}$ is an open subset of $X$ for all $j,l\ge1$ and
\[
Sh(f)=\bigcap_{l\ge1}\bigcup_{j\ge1}S_{j,l};
\]
therefore, $Sh(f)$ is a $G_{\delta}$-subset of $X$.
\item[(2)] A continuous map $f\colon X\to X$ is said to have the {\em shadowing property} if for any $\epsilon>0$, there is $\delta>0$ such that every $\delta$-pseudo orbit of $f$ is $\epsilon$-shadowed by some point of $X$, which is equivalent to $X=Sh(f)$ (see Lemma 2.4 of \cite{K1}).
\end{itemize}
\end{rem}

{\em Sensitivity} is a characteristic feature of chaotic dynamical systems. It is an element of some formal definitions of chaos and intuitively means that an arbitrarily small difference in initial conditions can be amplified to be a significant difference in later states (see, e.g., \cite{GW} for an in-depth look at the concept of sensitivity). The formal definition of sensitive points is as follows.

\begin{defi}
\normalfont
Given a continuous map $f\colon X\to X$ and $r>0$, $x\in X$ is called an {\em $r$-sensitive point} for $f$ if for any $\epsilon>0$, there are $y,z\in X$ and $i\ge0$ such that
\[
\max\{d(x,y),d(x,z)\}\le\epsilon
\]
and $d(f^i(y),f^i(z))>r$. We denote by $Sen_r(f)$ the set of $r$-sensitive points for $f$. We also define the set $Sen(f)$ of {\em sensitive points} for $f$ by
\[
Sen(f)=\bigcup_{r>0}Sen_r(f).
\]  
\end{defi}

For a continuous map $f\colon X\to X$, a subset $S$ of $X$ is said to be {\em $f$-invariant} if $f(S)\subset S$.

\begin{rem}
\normalfont
\begin{itemize}
\item[(1)] For any continuous map $f\colon X\to X$ and $r>0$, $Sen_r(f)$ is a closed $f$-invariant subset of $X$.
\item[(2)] A continuous map $f\colon X\to X$ is said to be {\em sensitive} if $X=Sen_r(f)$ for some $r>0$.
\end{itemize}
\end{rem}

Next, we recall the definition of entropy points. Note that the positive topological entropy is another characteristic feature of chaotic dynamical systems. The notion of entropy points is obtained by a concentration of positive topological entropy at a point \cite{YZ}. 

Let $f\colon X\to X$ be a continuous map. For $n\ge1$ and $r>0$, a subset $E$ of $X$ is said to be {\em $(n,r)$-separated} if 
\[
\max_{0\le i\le n-1}d(f^i(x),f^i(y))>r
\]
for all $x,y\in E$ with $x\ne y$. Let $K$ be a subset of $X$. For $n\ge1$ and $r>0$, let $s_n(f,K,r)$ denote the largest cardinality of an $(n,r)$-separated subset of $K$. We define $h(f,K,r)$ and $h(f,K)$ by
\[
h(f,K,r)=\limsup_{n\to\infty}\frac{1}{n}\log{s_n(f,K,r)}
\]
and
\[
h(f,K)=\lim_{r\to0}h(f,K,r).
\]
The topological entropy $h_{\rm top}(f)$ of $f$ is defined by $h_{\rm top}(f)=h(f,X)$.

\begin{defi}
\normalfont
Let $f\colon X\to X$ be a continuous map. For $x\in X$, let $\mathcal{K}(x)$ denote the set of closed neighborhoods of $x$.
\begin{itemize}
\item[(1)] $Ent(f)$ is the set of $x\in X$ such that $h(f,K)>0$ for all $K\in\mathcal{K}(x)$.
\item[(2)] For $r>0$, $Ent_r(f)$ is the set of $x\in X$ such that $h(f,K,r)>0$ for all $K\in\mathcal{K}(x)$.
\item[(3)] For $r,b>0$, $Ent_{r,b}(f)$ is the set of $x\in X$ such that $h(f,K,r)\ge b$ for all $K\in\mathcal{K}(x)$. 
\end{itemize}
\end{defi}

\begin{rem}
\normalfont
The following properties hold
\begin{itemize}
\item $Ent(f)$, $Ent_r(f)$, $r>0$, and $Ent_{r,b}(f)$, $r,b>0$, are closed $f$-invariant subsets of $X$,
\item
\[
Ent(f)\subset Ent_r(f)\subset Ent_{r,b}(f)
\]
for all $r,b>0$,
\item for any closed subset $K$ of $X$ and $r>0$, if $h(f,K,r)>0$, then $K\cap Ent_r(f)\ne\emptyset$,
\item for any closed subset $K$ of $X$ and $r,b>0$, if $h(f,K,r)\ge b$, then $K\cap Ent_{r,b}(f)\ne\emptyset$.
\end{itemize}
\end{rem}

\begin{rem}
\normalfont
Note that $Ent_r(f)\subset Sen_r(f)$ for all $r>0$.
\end{rem}

{\em Chain components}, which appear in the (so-called) fundamental theorem of dynamical systems by Conley, are essential objects for a global understanding of dynamical systems \cite{C}. Let us recall the definition of chain components.

\begin{defi}
\normalfont
Given a continuous map $f\colon X\to X$ and $\delta>0$, a finite sequence $(x_i)_{i=0}^{k}$ of points in $X$, where $k>0$ is a positive integer, is called a {\em $\delta$-chain} of $f$ if $d(f(x_i),x_{i+1})\le\delta$ for every $0\le i\le k-1$. For any $x,y\in X$, the notation $x\rightarrow y$ means that for every $\delta>0$,  there is a $\delta$-chain $(x_i)_{i=0}^k$ of $f$ with $x_0=x$ and $x_k=y$. The {\em chain recurrent set} $CR(f)$ for $f$ is defined by
\[
CR(f)=\{x\in X\colon x\rightarrow x\}.
\]
We define a relation $\leftrightarrow$ in
\[
CR(f)^2=CR(f)\times CR(f)
\]
by: for any $x,y\in CR(f)$, $x\leftrightarrow y$ if and only if $x\rightarrow y$ and $y\rightarrow x$. Note that $\leftrightarrow$ is a closed equivalence relation in $CR(f)^2$ and satisfies $x\leftrightarrow f(x)$ for all $x\in CR(f)$. An equivalence class $C$ of $\leftrightarrow$ is called a {\em chain component} for $f$. We denote by $\mathcal{C}(f)$ the set of chain components for $f$.
\end{defi}

\begin{rem}
\normalfont
The following properties hold
\begin{itemize}
\item $CR(f)=\bigsqcup_{C\in\mathcal{C}(f)}C$, a disjoint union,
\item every $C\in\mathcal{C}(f)$ is a closed $f$-invariant subset of $CR(f)$,
\item for all $C\in\mathcal{C}(f)$, $f|_C\colon C\to C$ is chain transitive, i.e., for any $x,y\in C$ and $\delta>0$, there is a $\delta$-chain $(x_i)_{i=0}^k$ of $f|_C$ with $x_0=x$ and $x_k=y$.
\end{itemize}
\end{rem}

Following \cite{AHK}, we define the {\em terminal} chain components as follows.

\begin{defi}
\normalfont
Given a continuous map $f\colon X\to X$, we say that a closed $f$-invariant subset $S$ of $X$ is {\em chain stable} if for any $\epsilon>0$, there is $\delta>0$ such that every $\delta$-chain $(x_i)_{i=0}^k$ of $f$ with $x_0\in S$ satisfies $d(x_k,S)=\inf_{y\in S}d(x_k,y)\le\epsilon$. We say that $C\in\mathcal{C}(f)$ is {\em terminal} if $C$ is chain stable. We denote by $\mathcal{C}_{\rm ter}(f)$ the set of terminal chain components for $f$.
\end{defi}

The following lemma is from \cite{K3}.

\begin{lem}[{\cite[Lemma 2.1]{K3}}]
Let $f\colon X\to X$ be a continuous map. For any $x\in X$, there are $C\in\mathcal{C}_{\rm ter}(f)$ and $y\in C$ such that $x\rightarrow y$.
\end{lem}

Given a continuous map $f\colon X\to X$ and $x\in X$, the {\em $\omega$-limit set} $\omega(x,f)$ of $x$ for $f$ is defined as the set of $y\in X$ such that
\[
\lim_{j\to\infty}f^{i_j}(x)=y
\]
for some sequence $0\le i_1<i_2<\cdots$. Note that $\omega(x,f)$ is a closed $f$-invariant subset of $X$ and satisfies
\[
\lim_{i\to\infty}d(f^i(x),\omega(x,f))=0.
\]
Since we have $y\rightarrow z$ for all $y,z\in\omega(x,f)$, there is an unique $C(x,f)\in\mathcal{C}(f)$ such that $\omega(x,f)\subset C(x,f)$ and so
\[
\lim_{i\to\infty}d(f^i(x),C(x,f))=0.
\]

The first result of this paper is the following theorem. For a subset $S$ of a topological space $Z$, $\overline{S}$ and ${\rm int}[S]$ denote the closure and the interior of $S$ respectively.

\begin{thm}
Let $f\colon X\to X$ be a continuous map. For any $x\in Sh(f)$ and $r>0$,
\begin{itemize}
\item[(1)] if $C(x,f)\in\mathcal{C}_{\rm ter}(f)$ and $x\in Sen_r(f)$, then $x\in Ent_s(f)$ for all $0<s<r$,
\item[(2)] if $x\in{\rm int}[Sen_r(f)]$, then $x\in Ent_s(f)$ for all $0<s<r$.
\end{itemize}
\end{thm}
 
The next lemma is a consequence of the Baire category theorem. The proof is left as an exercise for the reader.
 
\begin{lem} 
Let $Z$ be a complete metric space. Every sequence $A_j$, $j\ge1$, of closed subsets of $Z$ satisfies $\overline{{\rm int}[\bigcup_{j\ge1}A_j]}=\overline{\bigcup_{j\ge1}{\rm int}[A_j]}$.
\end{lem}

By Theorem 1.1 and Lemma 1.2, we obtain the following theorem.

\begin{thm}
Given a continuous map $f\colon X\to X$, if $X=\overline{Sh(f)}$, then
\[
\overline{{\rm int}[Sen(f)]}\subset Ent(f).
\]
\end{thm}

\begin{proof}
Letting $(r_j)_{j\ge1}$ be a sequence of numbers with $0<r_1>r_2>\cdots$ and $\lim_{j\to\infty}r_j=0$, we have
\[
Sen(f)=\bigcup_{j\ge1}Sen_{r_j}(f).
\]
We take a sequence $(s_j)_{j\ge1}$ of numbers such that $0<s_j<r_j$ for all $j\ge1$. By Theorem 1.1, we have
\[
{\rm int}[Sen_{r_j}(f)]\cap Sh(f)\subset Ent_{s_j}(f)
\]
for all $j\ge1$. Note that an open subset $U$ of $X$ and a subset $B$ of $X$ satisfies
\[
U\cap\overline{B}\subset\overline{U\cap B}.
\]
By $X=\overline{Sh(f)}$, we obtain
\[
{\rm int}[Sen_{r_j}(f)]={\rm int}[Sen_{r_j}(f)]\cap\overline{Sh(f)}\subset\overline{{\rm int}[Sen_{r_j}(f)]\cap Sh(f)}\subset \overline{Ent_{s_j}(f)}=Ent_{s_j}(f)
\]
for all $j\ge1$. With the aid of Lemma 1.2, we obtain
\[
\overline{{\rm int}[Sen(f)]}=\overline{{\rm int}[\bigcup\nolimits_{j\ge1}Sen_{r_j}(f)]}=\overline{\bigcup\nolimits_{j\ge1}{\rm int}[Sen_{r_j}(f)]}\subset\overline{\bigcup\nolimits_{j\ge1}Ent_{s_j}(f)}\subset Ent(f);
\]
thus, the theorem has been proved.
\end{proof}

Given a continuous map $f\colon X\to X$, we say that $x\in X$ is
\begin{itemize}
\item an {\em equicontinuity point} for $f$ if for any $\epsilon>0$, there is $\delta>0$ such that every $y\in X$ with $d(x,y)\le\delta$ satisfies
\[
\sup_{i\ge0}d(f^i(x),f^i(y))\le\epsilon,
\]
\item a {\em chain continuity point} for $f$ \cite{A} if for any $\epsilon>0$, there is $\delta>0$ such that every $\delta$-pseudo orbit $(x_i)_{i\ge0}$ of $f$ with $x_0=x$ is $\epsilon$-shadowed by $x$, i.e., satisfies
\[
\sup_{i\ge0}d(f^i(x),x_i)\le\epsilon.
\]
\end{itemize}
We denote by $EC(f)$ (resp.\:$CC(f)$) the set of equicontinuity (resp.\:chain continuity) points for $f$. By
\[
EC(f)=X\setminus Sen(f)=\bigcap_{n\ge1}[X\setminus Sen_{n^{-1}}(f)],
\]
we see that $EC(f)$ is a $G_\delta$-subset of $X$. It is easy to see that
\[
CC(f)=Sh(f)\cap EC(f)
\]
and so $CC(f)$ is a $G_\delta$-subset of $X$. We say that $f$ is {\em almost chain continuous} if $CC(f)$ is a dense ($G_\delta$-)subset of $X$.
\begin{thm}

For a continuous map $f\colon X\to X$, if $h_{\rm top}(f)=0$, then the following conditions are equivalent:
\begin{itemize}
\item[(A)] $X=\overline{Sh(f)}$,
\item[(B)] $f$ is almost chain continuous.
\end{itemize}
\end{thm}
\begin{proof}
The implication (A)$\implies$(B): Since $Sh(f)$ is a $G_\delta$-subset of $X$, if $X=\overline{Sh(f)}$, then $Sh(f)$ is a dense $G_\delta$-subset of $X$. Due to Theorem 1.2, if $X=\overline{Sh(f)}$ and
\[
{\rm int}[Sen(f)]\ne\emptyset,
\]
then $Ent(f)\ne\emptyset$ and so $h_{\rm top}(f)>0$.  By $X=\overline{Sh(f)}$ and $h_{\rm top}(f)=0$, we obtain
\[
{\rm int}[Sen(f)]=\emptyset;
\]
therefore, $EC(f)=X\setminus Sen(f)$ is a dense $G_\delta$-subset of $X$. By the Baire category theorem, we conclude that
\[
CC(f)=Sh(f)\cap EC(f)
\]
is a dense $G_\delta$-subset of $X$.

The implication (B)$\implies$(A) is a direct consequence of the inclusion $CC(f)\subset Sh(f)$.
\end{proof}

\begin{rem}
\normalfont
By Theorem 1.3, if a continuous map $f\colon X\to X$ satisfies $h_{\rm top}(f)=0$ and $X=Sh(f)$, i.e., the shadowing property, then $f$ is almost chain continuous. An alternative proof of this can be given as follows.

\begin{proof}
By Theorem A.1, for a continuous map $f\colon X\to X$ and $x\in X$, $x\in CC(f)$ holds exactly if $C(x,f)\in\mathcal{C}_{\rm ter}(f)$ and $CC(f|_{C(x,f)})=C(x,f)$. Every continuous map $f\colon X\to X$ satisfies $CR(f|_{CR(f)})=CR(f)$ and
\[
\mathcal{C}(f)=\mathcal{C}(f|_{CR(f)})=\mathcal{C}_{\rm ter}(f|_{CR(f)}).
\]
Assume that $h_{\rm top}(f)=0$ and $X=Sh(f)$. By Theorem B.1, $X=Sh(f)$ implies $CR(f)=Sh(f|_{CR(f)})$. If $Sen(f|_{CR(f)})\ne\emptyset$, by taking $x\in Sen(f|_{CR(f)})$ and $C\in\mathcal{C}(f)$ with $x\in C$, we obtain $C(x,f|_{CR(f)})=C$ and $x\in Sen_r(f|_{CR(f)})$ for some $r>0$. By $C\in\mathcal{C}_{\rm ter}(f|_{CR(f)})$ and Theorem 1.1, we obtain $x\in Ent_s(f|_{CR(f)})$ for all $0<s<r$ and so $h_{\rm top}(f)>0$, a contradiction. It follows that $Sen(f|_{CR(f)})=\emptyset$; thus, $EC(f|_{CR(f)})=CR(f)$. We obtain
\[
CC(f|_{CR(f)})=Sh(f|_{CR(f)})\cap EC(f|_{CR(f)})=CR(f)
\]
(cf.\:\:Corollary 6 of \cite{Moo}). This implies that $CC(f|_C)=C$ for all $C\in\mathcal{C}(f)$. As a consequence, if $h_{\rm top}(f)=0$ and $X=Sh(f)$, then
\[
CC(f)=\{x\in X\colon C(x,f)\in\mathcal{C}_{\rm ter}(f)\}.
\]
On the other hand, by Theorem 1.1 of \cite{K5}, $X=Sh(f)$ implies that
\[
\{x\in X\colon C(x,f)\in\mathcal{C}_{\rm ter}(f)\}
\]
is a dense $G_{\delta}$-subset $X$. By these conditions, we conclude that $CC(f)$ is a dense $G_{\delta}$-subset $X$, i.e., $f$ is almost chain continuous.
\end{proof}
\end{rem}

Let $f\colon X\to X$ be a continuous map. For $\delta,r>0$ and $n\ge1$, we say that two $\delta$-chains $(x_i)_{i=0}^n$, $(y_i)_{i=0}^n$ of $f$ is {\em $(n,r)$-separated} if $d(x_i,y_i)>r$ for some $0\le i\le n$. Let
\[
s_n(f,X,r,\delta)
\]
denote the largest cardinality of a set of pairwise $(n,r)$-separated $\delta$-chains of $f$. The next lemma is from \cite{Mis}.

\begin{lem}[Misiurewicz]
\[
h_{\rm top}(f)=\lim_{r\to0}\lim_{\delta\to0}\limsup_{n\to\infty}\frac{1}{n}\log{s_n(f,X,r,\delta)}.
\]
\end{lem}

Let $f\colon X\to X$ be a continuous map. Following \cite{YZ}, we define the set $Ent_{up}(f)$ of {\em uniform entropy points} for $f$ by
\[
Ent_{up}(f)=\bigcup_{r,b>0}Ent_{r,b}(f).
\]
Note that
\[
Ent(f)=\overline{Ent_{up}(f)}.
\]

The following lemma can be proved by Lemma 1.3 and a similar argument as in the proof of Theorem 1.2 in \cite{K4}. Note that the `only if' part is rather trivial.

\begin{lem}
For a continuous map $f\colon X\to X$ and $C\in\mathcal{C}_{\rm ter}(f)$, $h_{\rm top}(f|_C)>0$ holds if and only if $C\cap Ent_{up}(f)\ne\emptyset$.
\end{lem}

By using Lemma 1.1 and Lemma 1.4, we obtain the following theorem.

\begin{thm}
For a continuous map $f\colon X\to X$, if $X=Sh(f)$, then the following conditions are equivalent:
\begin{itemize}
\item[(A)] $X=Ent_{up}(f)$,
\item[(B)] $h_{\rm top}(f|_C)>0$ for all $C\in\mathcal{C}_{\rm ter}(f)$.
\end{itemize}
\end{thm}

\begin{proof}
The implication (A)$\implies$(B) is a direct consequence of Lemma 1.4.

The implication (B)$\implies$(A): Due to Lemma 1.1, for any $x\in X$, there are $C\in\mathcal{C}_{\rm ter}(f)$ and $y\in C$ such that $x\rightarrow y$. By $h_{\rm top}(f|_C)>0$, we obtain $z\in Ent_{r,b}(f)$ for some $z\in C$ and $r,b>0$. Since $y,z\in C$, we have $y\rightarrow z$. This with $x\rightarrow y$ implies $x\rightarrow z$. By $x\in Sh(f)$ and Lemma 2.1 in Section 2, we obtain $x\in Ent_{s,b}(f)$ for all $0<s<r$; thus, $x\in Ent_{up}(f)$. Since $x\in X$ is arbitrary, we conclude that $X=Ent_{up}(f)$.
\end{proof}

We recall the definition of C-entropy points from \cite{YZ}.  For a continuous map $f\colon X\to X$ and an open cover $\mathcal{U}$ of $X$, let $h(f,\mathcal{U})$ denote the entropy of $f$ relative to $\mathcal{U}$ (see \cite{W} for details). We say that $(x,y)\in X^2$ with $x\ne y$ is an {\em entropy pair} for $f$ if for any closed neighborhoods $A$ of $x$ and $B$ of $y$,
\[
h(f,\{X\setminus A,X\setminus B\})>0
\]
whenever $A\cap B=\emptyset$. We denote by $E_2(f)$ the set of entropy pairs for $f$. We say that $x\in X$ is a {\em C-entropy point} for $f$ if $(x,y)\in E_2(f)$ for some $y\in X$ with $x\ne y$. Let $E_1(f)$ denote the set of C-entropy points for $f$. By Theorem 3.4 of \cite{YZ}, we know that $E_1(f)\subset Ent(f)$. In \cite{YZ}, Ye and Zhang raised the following question.

\begin{ques}[{\cite[Question 6.11]{YZ}}]
Does $E_1(f)\subset Ent_{up}(f)$ hold for any homeomorphism $f\colon X\to X$?
\end{ques}

In Section 3, we present an example of a homeomorphism $f\colon X\to X$ with the following properties

\begin{itemize}
\item[(1)] $X=Sh(f)$,
\item[(2)] there exists $C\in\mathcal{C}_{\rm ter}(f)$ such that $C$ is an infinite set and satisfies $C\subset E_1(f)\setminus Ent_{up}(f)$.
\end{itemize}

This paper consists of three sections and two appendices. In the next section, we prove Theorem 1.1. In Section 3, we present the example just mentioned above. In Appendix A, we prove Theorem A.1. In Appendix B, Theorem B.1 is proved. 

\section{Proof of Theorem 1.1}

In this section, we prove Theorem 1.1. For the proof, we need a few lemmas. The following lemma is proved in \cite{K4} (see Lemma 1.2 and Lemma 3.1 of \cite{K4} where $C(x)$ is defined by $C(x)=\{x\}\cup\{y\in X\colon x\rightarrow y\}$).

\begin{lem}
Let $f\colon X\to X$ be a continuous map. Let $x,y\in X$ and $r,b>0$. If $x\rightarrow y$, then
\begin{itemize}
\item[(1)] $x\in Sh(f)$ implies $y\in Sh(f)$,
\item[(2)] $x\in Sh(f)$ and  $y\in Ent_r(f)$ imply $x\in Ent_s(f)$ for all $0<s<r$,
\item[(3)] $x\in Sh(f)$ and  $y\in Ent_{r,b}(f)$ imply $x\in Ent_{s,b}(f)$ for all $0<s<r$.
\end{itemize}
\end{lem}

Let $f\colon X\to X$ be a continuous map. For a closed $f$-invariant subset $S$ of $X$ and $r>0$, we denote by $Sen_r^\ast(f|_S)$ the set of $x\in S$ such that for any $\delta>0$, there are $\delta$-chains $(x_i)_{i=0}^k$, $(y_i)_{i=0}^k$ of $f|_S$ with $x_0=y_0=x$ and $d(x_k,y_k)>r$. For $x\in X$ and $\epsilon>0$, let $B_\epsilon(x)$ denote the closed $\epsilon$-ball centered at $x$: $B_\epsilon(x)=\{y\in X\colon d(x,y)\le\epsilon\}$.

\begin{lem} 
Let $f\colon X\to X$ be a continuous map and let $C\in\mathcal{C}(f)$. For any $x\in C$ and $r>0$, if
\[
x\in Sh(f)\cap Sen_r^\ast(f|_C),
\]
then  $x\in Ent_s(f)$ for all $0<s<r$.
\end{lem}

\begin{proof}
We fix $s,\epsilon>0$ such that $s+2\epsilon<r$. By $x\in Sh(f)$, we obtain $\delta>0$ such that every $\delta$-pseudo orbit $(x_i)_{i\ge0}$ of $f$ with $x_0=x$ is $\epsilon$-shadowed by some $y\in X$. Since $x\in Sen_r^\ast(f|_C)$, there is a pair
\[
(\alpha_0,\alpha_1)=((x_i^{(0)})_{i=0}^k,(x_i^{(1)})_{i=0}^k)
\]
of $\delta$-chains of $f|_C$ such that $x_0^{(0)}=x_0^{(1)}=x$ and $d(x_k^{(0)},x_k^{(1)})>r$. Since $x,x_k^{(0)},x_k^{(1)}\in C$, we have $x_k^{(0)}\rightarrow x$ and $x_k^{(1)}\rightarrow x$; therefore, there is a pair
\[
(\beta_0,\beta_1)=((y_i^{(0)})_{i=0}^l,(y_i^{(1)})_{i=0}^m)
\]
of $\delta$-chains of $f$ such that $y_0^{(0)}=x_k^{(0)}$, $y_0^{(1)}=x_k^{(1)}$, and $y_l^{(0)}=y_m^{(1)}=x$. Let $n=2k+l+m$ and let
\[
(\gamma_0,\gamma_1)=(\alpha_0\beta_0\alpha_1\beta_1,\alpha_1\beta_1\alpha_0\beta_0)=((z_i^{(0)})_{i=0}^n,(z_i^{(1)})_{i=0}^n),
\]
a pair of $\delta$-chains of $f$. Note that $z_0^{(0)}=z_0^{(1)}=z_n^{(0)}=z_n^{(1)}=x$ and
\[
d(z_k^{(0)},z_k^{(1)})=d(x_k^{(0)},x_k^{(1)})>r.
\]
For any $N\ge1$ and $u=(u_j)_{j=1}^N\in\{0,1\}^N$, let
\[
\gamma_u=(w_i^{(u)})_{i=0}^{nN}=\gamma_{u_1}\gamma_{u_2}\cdots\gamma_{u_N},
\]
a $\delta$-chain of $f$. Since $w_0^{(u)}=x$, by the choice of $\delta$, we have $x_u\in X$ such that
\[
d(f^i(x_u),w_i^{(u)})\le\epsilon
\]
for all $0\le i\le nN$. Note that
\[
d(x,x_u)=d(w_0^{(u)},x_u)\le\epsilon
\]
and so $x_u\in B_\epsilon(x)$ for all $u\in\{0,1\}^N$. For any $u,v\in\{0,1\}^N$, $u\ne v$ implies $u_j\ne v_j$ and so
\begin{align*}
d(f^{k+(j-1)n}(x_u),f^{k+(j-1)n}(x_v))&\ge d(w_{k+(j-1)n}^{(u)},w_{k+(j-1)n}^{(v)})-2\epsilon \\
&=d(z_k^{(0)},z_k^{(1)})-2\epsilon>r-2\epsilon>s
\end{align*}
for some $1\le j\le N$. It follows that $\{x_u\colon u\in\{0,1\}^N\}$ is an $(nN,s)$-separated subset of $B_\epsilon(x)$ with
\[
|\{x_u\colon u\in\{0,1\}^N\}|=2^N.
\]
Since $N\ge1$ is arbitrary, we obtain
\begin{align*}
h(f,B_\epsilon(x),s)&\ge\limsup_{N\to\infty}\frac{1}{nN}\log{s_{nN}(f,B_\epsilon(x),s)}\\
&\ge\limsup_{N\to\infty}\frac{1}{nN}\log{2^N}=\frac{1}{n}\log{2}>0.
\end{align*}
Since $\epsilon>0$ with $s+2\epsilon<r$ is arbitrary, we conclude that $x\in Ent_s(f)$, completing the proof of the lemma.
\end{proof} 

\begin{lem} 
Let $f\colon X\to X$ be a continuous map and let $C\in\mathcal{C}_{\rm ter}(f)$. For any $x\in C$ and $r>0$, $x\in Sen_r(f)$ implies $x\in Sen_s^\ast(f|_C)$ for all $0<s<r$.
\end{lem}

\begin{proof}
Since $x\in C$ and $C\in\mathcal{C}_{\rm ter}(f)$, for any $\epsilon>0$, there is $\gamma>0$ such that every $\gamma$-chain $(x_i)_{i=0}^k$ of $f$ with $x_0=x$ satisfies
\[
d(x_i,C)\le\epsilon
\]
for all $0\le i\le k$. By $x\in Sen_r(f)$, we obtain a pair
\[
((x_i^{(0)})_{i=0}^k,(x_i^{(1)})_{i=0}^k)
\]
of $\gamma$-chains of $f$ with $x_0^{(0)}=x_0^{(1)}=x$ and $d(x_k^{(0)},x_k^{(1)})>r$. For each $j\in\{0,1\}$, we take a sequence $y_i^{(j)}\in C$, $0\le i\le k$, with $y_0^{(j)}=x_0^{(j)}=x$ and
\[
d(x_i^{(j)},y_i^{(j)})=d(x_i^{(0)},C)\le\epsilon
\]
for all $0<i\le k$. For every $\delta>0$, if $\epsilon,\gamma>0$ are sufficiently small, then
\[
((y_i^{(0)})_{i=0}^k,(y_i^{(1)})_{i=0}^k)
\]
is a pair of $\delta$-chains of $f|_C$ such that $y_0^{(0)}=y_0^{(1)}=x$ and $d(y_k^{(0)},y_k^{(1)})>s$. This implies $x\in Sen_s^\ast(f|_C)$; therefore, the lemma has been proved.
\end{proof}

By using these lemmas, we prove Theorem 1.1.

\begin{proof}[Proof of Theorem 1.1]
(1): We take $\epsilon>0$ such that $s+2\epsilon<r$. Let $y\in\omega(x,f)$. Since $Sen_r(f)$ is a closed $f$-invariant subset of $X$, $x\in Sen_r(f)$ implies $\omega(x,f)\subset Sen_r(f)$. Note that
\[
\omega(x,f)\subset C(x,f).
\]
By $C(x,f)\in\mathcal{C}_{\rm ter}(f)$ and Lemma 2.3, we obtain
\[
y\in Sen_r(f)\cap C(x,f)\subset Sen_{s+2\epsilon}^\ast(f|_{C(x,f)}).
\]
Note that $y\in\omega(x,f)$ implies $x\rightarrow y$. By $x\in Sh(f)$ and Lemma 2.1, we obtain $y\in Sh(f)$ and so
\[
y\in Sh(f)\cap Sen_{s+2\epsilon}^\ast(f|_{C(x,f)}).
\]
This with Lemma 2.2 implies $y\in Ent_{s+\epsilon}(f)$. By $x\rightarrow y$, $x\in Sh(f)$, and Lemma 2.1, we conclude that $x\in Ent_s(f)$.

(2): Due to Lemma 1.1, there are $C\in\mathcal{C}_{\rm ter}(f)$ and $y\in C$ such that $x\rightarrow y$. Let $(\epsilon_j)_{j\ge1}$ be a sequence of numbers with $0<\epsilon_1>\epsilon_2 >\cdots$ and $\lim_{j\to\infty}\epsilon_j=0$. By $x\in Sh(f)$, for each $j\ge1$, we have $\delta_j>0$ such that every $\delta_j$-pseudo orbit $(x_i)_{i\ge0}$ of $f$ with $x_0=x$ is $\epsilon_j$-shadowed by some $z\in X$. Since $x\rightarrow y$, for each $j\ge1$, there is a $\delta_j$-chain $(x_i^{(j)})_{i=0}^{k_j}$ of $f$ with $x_0^{(j)}=x$ and $x_{k_j}^{(j)}=y$. We obtain a sequence $x_j\in X$, $j\ge1$, such that
\[
d(f^i(x_j),x_i^{(j)})\le\epsilon_j
\]
for all $j\ge1$ and $0\le i\le k_j$. Note that
\[
d(x_j,x)=d(x_j,x_0^{(j)})\le\epsilon_j
\]
and
\[
d(f^{k_j}(x_j),y)=d(f^{k_j}(x_j),x_{k_j}^{(j)})\le\epsilon_j
\]
for all $j\ge1$. Since $x\in{\rm int}[Sen_r(f)]$, we have $x_j\in Sen_r(f)$ for all $j$ large enough. Since $Sen_r(f)$ is a closed $f$-invariant subset of $X$, it follows that $f^{k_j}(x_j)\in Sen_r(f)$ for all $j$ large enough; thus, by $\lim_{j\to\infty}f^{k_j}(x_j)=y$, we obtain
\[
y\in Sen_r(f)\cap C.
\]
Note that $C\in\mathcal{C}_{\rm ter}(f)$, $x\rightarrow y$, and $x\in Sh(f)$. The rest of the proof is similar to that of (1). 
\end{proof}

\section{Example}

In this section, we present an example of a homeomorphism $f\colon X\to X$ with the following properties
\begin{itemize}
\item[(1)] $X=Sh(f)$,
\item[(2)] there exists $C\in\mathcal{C}_{\rm ter}(f)$ such that $C$ is an infinite set and satisfies $C\subset E_1(f)\setminus Ent_{up}(f)$.
\end{itemize}

The following example is taken from \cite{K3}. We modify Example 4.3 of \cite{K3} to obtain the homeomorphism.

\begin{ex}
\normalfont
Let $\sigma\colon[-1,1]^\mathbb{Z}\to[-1,1]^\mathbb{Z}$ be the shift map, i.e.,
\[
\sigma(x)_n=x_{n+1}
\]
for all $x=(x_n)_{n\in\mathbb{Z}}$ and $n\in\mathbb{Z}$. Let $d$ be the metric on $[-1,1]^\mathbb{Z}$ defined by
\[
d(x,y)=\sup_{n\in\mathbb{Z}}2^{-|n|}|x_n-y_n|
\]
for all $x=(x_n)_{n\in\mathbb{Z}},y=(y_n)_{n\in\mathbb{Z}}\in[-1,1]^\mathbb{Z}$. We fix a sequence $s=(s_k)_{k\ge 1}$ of numbers with $0<s_1<s_2<\cdots$ and $\lim_{k\to\infty}s_k=1$. Put
\[
S=\{-1,1\}\cup\{-s_k\colon k\ge1\}\cup\{s_k\colon k\ge1\},
\]
a closed subset of $[-1,1]$.

We define a closed $\sigma$-invariant subset $X$ of $S^\mathbb{Z}$ by for any $x=(x_n)_{n\in\mathbb{Z}}\in S^\mathbb{Z}$, $x\in X$ if and only if the following conditions are satisfied
\begin{itemize}
\item $|x_n|\le|x_{n+1}|$ for all $n\in\mathbb{Z}$,
\item for any $n\in\mathbb{Z}$ and $k\ge1$, if $x_n=s_k$, then $x_{n+j}=-s_k$ for all $1\le j\le k$,
\item for every $n\in\mathbb{Z}$, if $x_n=1$, then $x_{n+j}=-1$ for all $j\ge1$.
\end{itemize}
We define $z=(z_n)_{n\in\mathbb{Z}}$ and $z^{(m)}=(z_n^{(m)})_{n\in\mathbb{Z}}$, $m\in\mathbb{Z}$, as follows.
\begin{itemize}
\item $z_n=-1$ for all $n\in\mathbb{Z}$,
\item $z_m^{(m)}=1$ for all $m\in\mathbb{Z}$,
\item $z_n^{(m)}=-1$ for all $m,n\in\mathbb{Z}$ with $m\ne n$.
\end{itemize}
Let $f=\sigma|_X\colon X\to X$, $X_k=X\cap\{-s_k,s_k\}^\mathbb{Z}$, $k\ge1$, and let
\[
X_\infty=X\cap\{-1,1\}^\mathbb{Z}=\{z\}\cup\{z^{(m)}\colon m\in\mathbb{Z}\}.
\]
A similar argument as in Example 4.3 of \cite{K3} shows that
\[
CR(f)=\{x=(x_n)_{n\in\mathbb{Z}}\in X\colon|x_n|=|x_{n+1}|\:\:\text{for all $n\in\mathbb{Z}$}\}=X_\infty\cup\bigcup_{k\ge1}X_k
\]
and $\mathcal{C}(f)=\{X_\infty\}\cup\{X_k\colon k\ge1\}$. Note that $\mathcal{C}_{\rm ter}(f)=\{X_\infty\}$.

Similarly as in Example 4.3 of \cite{K3}, by taking a sequence $s=(s_k)_{k\ge 1}$ that rapidly converges to $1$, we can show that $X=Sh(f)$, i.e., $f$ has the shadowing property. Since $X_\infty\in\mathcal{C}_{\rm ter}(f)$ and $h_{\rm top}(f|_{X_\infty})=0$, by Lemma  
1.4, we obtain
\[
X_\infty\cap Ent_{up}(f)=\emptyset.
\]
Note that the sequence $X_k$, $k\ge1$, satisfies
\[
\lim_{k\to\infty}X_k=X_\infty
\]
with respect to the Hausdorff distance. For every $k\ge1$, since $X_k$ is a mixing subshift of finite type, $f|_{X_k}$ is mixing and satisfies the shadowing property. This implies that for every $k\ge1$, $f|_{X_k}$ has uniformly positive entropy (upe) and so satisfies
\[
(x,y)\in E_2(f)  
\]
for all $x,y\in X_k$ with $x\ne y$ (see, e.g., \cite{K2} for details). Given any $u,v\in X_\infty$ with $u\ne v$, by taking two sequences $(x_k)_{k\ge1}$, $(y_k)_{k\ge1}$ such that
\begin{itemize}
\item $x_k,y_k\in X_k$ and $x_k\ne y_k$ for all $k\ge1$,
\item $\lim_{k\to\infty}x_k=u$ and $\lim_{k\to\infty}y_k=v$,
\end{itemize}
we obtain $(u,v)\in E_2(f)$. It follows that $X_\infty\subset E_1(f)$; thus, this example satisfies the desired properties. 
\end{ex}

\appendix

\section{}

In this Appendix A, we prove the following theorem.
\begin{thm}
Let $f\colon X\to X$ be a continuous map. Let $x\in X$ and $C=C(x,f)$. The following conditions are equivalent:
\begin{itemize}
\item[(A)] $x\in CC(f)$,
\item[(B)] $C\in\mathcal{C}_{\rm ter}(f)$ and $C\subset CC(f)$,
\item[(C)] $C\in\mathcal{C}_{\rm ter}(f)$ and $CC(f|_C)=C$,
\item[(D)] $C\in\mathcal{C}_{\rm ter}(f)$ and $CC(f|_C)\ne\emptyset$.
\end{itemize}
\end{thm}

For the proof, we need several lemmas.

\begin{lem}
For a continuous map $f\colon X\to X$ and $x,y\in X$, if $x\in CC(f)$ and $x\rightarrow y$, then $y\in CC(f)$.
\end{lem}

\begin{proof}
Let $\epsilon>0$. By $x\in CC(f)$, we have $\delta>0$ such that every $\delta$-pseudo orbit $(x_i)_{i\ge0}$ of $f$ with $x_0=x$ is $\epsilon$-shadowed by $x$. Let $\xi=(y_i)_{i\ge0}$ be a $\delta$-pseudo orbit of $f$ with $y_0=y$. We fix a sequence $(\delta_j)_{j\ge1}$ of numbers with $\delta\ge\delta_1>\delta_2>\cdots$ and $\lim_{j\to\infty}\delta_j=0$. Since $x\rightarrow y$, for every $j\ge1$, there is a $\delta_j$-chain $\alpha_j=(x_i^{(j)})_{i=0}^{k_j}$ of $f$ with $x_0^{(j)}=x$ and $x_{k_j}^{(j)}=y$. By $x\in CC(f)$ and $\lim_{j\to\infty}\delta_j=0$, we obtain $\lim_{j\to\infty}f^{k_j}(x)=y$. For each $j\ge1$, let
\[
\xi_j=(y_i^{(j)})_{i\ge0}=\alpha_j\xi,
\]
a $\delta$-pseudo orbit of $f$. Given any $j\ge1$, since $y_0^{(j)}=x_0^{(j)}=x$, we have
\[
\sup_{i\ge0}d(f^i(x),y_i^{(j)})\le\epsilon
\]
which implies
\[
\sup_{i\ge0}d(f^i(f^{k_j}(x)),y_i)=\sup_{i\ge0}d(f^{i+k_j}(x),y_{i+k_j}^{(j)})\le\epsilon.
\]
Letting $j\to\infty$, we obtain
\[
\sup_{i\ge0}d(f^i(y),y_i)\le\epsilon.
\]
Since $\xi$ is arbitrary, we conclude that $y\in CC(f)$, proving the lemma.
\end{proof}

\begin{lem}
For a continuous map  $f\colon X\to X$ and $C\in\mathcal{C}(f)$, if $C\cap CC(f)\ne\emptyset$, then $C\in\mathcal{C}_{\rm ter}(f)$.
\end{lem}

\begin{proof}
Let $x\in C\cap CC(f)$. Let $y\in C$, $z\in X$ and $y\rightarrow z$. By $x,y\in C$, we have $x\rightarrow y$. By $x\rightarrow y$ and $y\rightarrow z$, we obtain $x\rightarrow z$. Since $x\in CC(f)$, this implies $z\in C$; thus, we conclude that $C\in\mathcal{C}_{\rm ter}(f)$.
\end{proof}

\begin{lem}
Let $f\colon X\to X$ be a continuous map and let $C\in\mathcal{C}(f)$. If $C\in\mathcal{C}_{\rm ter}(f)$, then $CC(f|_C)\subset CC(f)$.
\end{lem}

\begin{proof}
Let $x\in CC(f|_C)$ and $\epsilon>0$. By $x\in CC(f|_C)$, we have $\delta>0$ such that every $\delta$-pseudo orbit $(x_i)_{i\ge0}$ of $f|_C$ with $x_0=x$ is $\epsilon$-shadowed by $x$. Let $0<\beta\le\epsilon$. Since $C\in\mathcal{C}_{\rm ter}(f)$ and $x\in C$, there is $\gamma>0$ such that every $\gamma$-pseudo orbit $\xi=(y_i)_{i\ge0}$ of $f$ with $y_0=x$ satisfies
\[
\sup_{i\ge0}d(y_i,C)\le\beta.
\]
If $\beta,\gamma>0$ are sufficiently small, then by letting $x_0=x$ and taking $x_i\in C$, $i>0$, such that
\[
d(y_i,x_i)=d(y_i,C)\le\beta
\]
for all $i>0$, we obtain a $\delta$-pseudo orbit $\xi'=(x_i)_{i\ge0}$ of $f|_C$ with $x_0=x$.  It follows that
\[
\sup_{i\ge0}d(f^i(x),x_i)\le\epsilon
\]
and so
\[
\sup_{i\ge0}d(f^i(x),y_i)\le\epsilon+\beta\le2\epsilon.
\] 
Since $\xi$ is arbitrary, we obtain $x\in CC(f)$ and thus $CC(f|_C)\subset CC(f)$, proving the lemma.
\end{proof}

\begin{lem}
Let $f\colon X\to X$ be a continuous map. Let $x\in X$ and let $C=C(x,f)$. If $C\in\mathcal{C}_{\rm ter}(f)$ and $CC(f|_C)\ne\emptyset$, then $x\in CC(f)$.
\end{lem}

\begin{proof}
Let $y\in CC(f|_C)$ and $z\in\omega(x,f)$. By $C\in\mathcal{C}_{\rm ter}(f)$ and Lemma A.3, we obtain $y\in CC(f)$. Since $\omega(x,f)\subset C$, we have $y,z\in C$ and so $y\rightarrow z$. By Lemma A.1, we obtain $z\in CC(f)$. Then, for any $\epsilon>0$, there is $\delta>0$ such that every $\delta$-pseudo orbit $(x_i)_{i\ge0}$ of $f$ with $x_0=z$ is $\epsilon$-shadowed by $z$. For $0<\beta\le2\epsilon$, we fix $k>0$ with $d(z,f^k(x))\le\beta$. By continuity of $f$, we have $\gamma>0$ such that every $\gamma$-pseudo orbit $\xi=(y_i)_{i\ge0}$ of $f$ with $y_0=x$ satisfies
\[
\sup_{0\le i\le k}d(f^i(x),y_i)\le\beta
\] 
and so
\[
d(z,y_k)\le d(z,f^k(x))+d(f^k(x),y_k)\le2\beta.
\]
If $\beta,\gamma>0$ are sufficiently small, then
\[
(z,f^{k+1}(x),f^{k+2}(x),\dots)
\]
and
\[
(z,y_{k+1},y_{k+2},\dots)
\]
are $\delta$-pseudo orbits of $f$ and so $\epsilon$-shadowed by $z$, which implies
\[
\sup_{i\ge k+1}d(f^i(x),y_i)\le2\epsilon.
\]
Note that
\[
\sup_{0\le i\le k}d(f^i(x),y_i)\le\beta\le2\epsilon.
\] 
Since $\xi$ is arbitrary, we conclude that $x\in CC(f)$, completing the proof.
\end{proof}

By these lemmas, we prove Theorem A.1.

\begin{proof}[Proof of Theorem A.1]
The implication (A)$\implies$(B) is a consequence of Lemma A.1 and Lemma A.2. (B)$\implies$(C) and (C)$\implies$(D) are obvious by definition. (D)$\implies$(A) is a consequence of Lemma A.4. This completes the proof of the theorem.
\end{proof}

\section{}
 
The aim of this Appendix B is to prove the following theorem.

\begin{thm}
Every continuous map $f\colon X\to X$ satisfies $Sh(f)\cap CR(f)\subset Sh(f|_{CR(f)})$.
\end{thm}

Let $f\colon X\to X$ be a continuous map. For $x,y\in X$ and $\delta>0$, the notation $x\rightarrow_\delta y$ means that there is a $\delta$-chain $(x_i)_{i=0}^k$ of $f$ with $x_0=x$ and $x_k=y$. We define an equivalence relation $\leftrightarrow_\delta$ in
\[
CR(f)^2=CR(f)\times CR(f)
\]
by: for any $x,y\in CR(f)$, $x\leftrightarrow_\delta y$ if and only if $x\rightarrow_\delta y$ and $y\rightarrow_\delta x$. It is easy to see that $x\leftrightarrow_\delta f(x)$ for all $x\in CR(f)$. It is also not difficult to see that $x\leftrightarrow_\delta y$ for all $x,y\in CR(f)$ with $d(x,y)\le\delta$.

\begin{proof}[Proof of Theorem B.1]
Let $x\in Sh(f)\cap CR(f)$. Then, for any $\epsilon>0$, we have $\delta>0$ such that every $\delta$-pseudo orbit $(x_i)_{i\ge0}$ of $f$ with $x_0=x$ is $\epsilon$-shadowed by some $y\in X$. Let $\xi=(x_i)_{i\ge0}$ be a $\delta$-pseudo orbit of $f|_{CR(f)}$ with $x_0=x$. For every $i\ge0$, since $d(f(x_i),x_{i+1})\le\delta$, we have $x_i\leftrightarrow_\delta f(x_i)$ and $f(x_i)\leftrightarrow_\delta x_{i+1}$; therefore, $x_i\leftrightarrow_\delta x_{i+1}$. This implies that $x_0\leftrightarrow_\delta x_k$ for all $k>0$. Then, for any $k>0$, there is a $\delta$-chain $(y_i)_{i=0}^l$ of $f$ with $y_0=x_k$ and $y_l=x_0$. Since
\[
\xi_k=(x_0,x_1,\dots,x_{k-1},y_0,y_1,\dots,y_{l-1},x_0,x_1,\dots,x_{k-1},y_0,y_1,\dots,y_{l-1},\dots)
\]
is a $\delta$-pseudo orbit of $f$ with $x_0=x$, letting
\[
X_k=\{x\in X\colon\text{$\xi_k$ is $\epsilon$-shadowed by $x$}\},
\]
we have $X_k\ne\emptyset$. Note that $X_k$ is a closed $f^{k+l}$-invariant subset of $X$. By taking $u\in X_k$ and $v\in\omega(u,f^{k+l})$, we obtain $v\in X_k\cap CR(f^{k+l})\subset X_k\cap CR(f)$ and thus
\[
CR(f)\cap\bigcap_{i=0}^{k-1}f^{-i}(B_\epsilon(x_i))\ne\emptyset.
\]
It follows that
\[
CR(f)\cap\bigcap_{i=0}^\infty f^{-i}(B_\epsilon(x_i))=\bigcap_{k>0}\left(CR(f)\cap\bigcap_{i=0}^{k-1}f^{-i}(B_\epsilon(x_i))\right)\ne\emptyset,
\]
i.e., $\xi$ is $\epsilon$-shadowed by some $y\in CR(f)$. Since $\xi$ is arbitrary, we obtain $x\in Sh(f|_{CR(f)})$, proving the theorem.
\end{proof}


\begin{thebibliography}{99}

\bibitem{A} E.\:Akin, On chain continuity. Discrete Contin. Dyn. Syst. 2 (1996), 111--120.

\bibitem{AHK} E.\:Akin, M.\:Hurley, J.\:Kennedy, Dynamics of topologically generic homeomorphisms. Mem. Amer. Math. Soc. 164 (2003).

\bibitem{An} D.V.\:Anosov, Geodesic flows on closed Riemann manifolds with negative curvature. Proc. Steklov Inst. Math. 90 (1967), 235 p.

\bibitem{AH} N.\:Aoki, K.\:Hiraide, Topological theory of dynamical systems. Recent advances. North--Holland Mathematical Library, 52. North--Holland Publishing Co., 1994.

\bibitem{B} R.\:Bowen, Equilibrium states and the ergodic theory of Anosov diffeomorphisms. Lecture Notes in Mathematics, 470. Springer--Verlag, 1975.

\bibitem{C} C.\:Conley, Isolated invariant sets and the Morse index. CBMS Regional Conference Series in Mathematics, 38. American Mathematical Society, Providence, R.I., 1978.

\bibitem{GW} E.\:Glasner, B.\:Weiss, Sensitive dependence on initial conditions, Nonlinearity 6 (1993), 1067--1075.

\bibitem{K1} N.\:Kawaguchi, Quantitative shadowable points. Dyn. Syst. 32 (2017), 504--518.

\bibitem{K2} N.\:Kawaguchi, Distributionally chaotic maps are $C^0$-dense. Proc. Amer. Math. Soc. 147 (2019), 5339--5348.

\bibitem{K3} N.\:Kawaguchi, Generic and dense distributional chaos with shadowing. J. Difference Equ. Appl. 27 (2021), 1456--1481.

\bibitem{K4} N.\:Kawaguchi, Some results on shadowing and local entropy properties of dynamical systems. Bull. Braz. Math. Soc. (N.S.) 55 (2024), 18, 29 pp.

\bibitem{K5} N.\:Kawaguchi, Shadowing and the basins of terminal chain components. Can. Math. Bull. 68 (2025), 187--197.

\bibitem{Mis} M.\:Misiurewicz, Remark on the definition of topological entropy. Dynamical systems and partial differential equations (Caracas, 1984), Univ. Simon Bolivar, Caracas, 1986, 65--67.

\bibitem{Moo} T.K.S.\:Moothathu, Implications of pseudo-orbit tracing property for continuous maps on compacta. Topology Appl. 158 (2011), 2232--2239.

\bibitem{Mor} C.A.\:Morales, Shadowable points. Dyn. Syst. 31 (2016), 347--356.

\bibitem{W} P.\:Walters, An introduction to ergodic theory. Graduate Texts in Mathematics, 79. Springer, 1982.

\bibitem{YZ} X.\:Ye, G.\:Zhang, Entropy points and applications. Trans. Amer. Math. Soc. 359 (2007), 6167--6186.

\end{thebibliography}
\end{document}